\documentclass[11pt,reqno,english]{amsart}

\usepackage[foot]{amsaddr}
\usepackage[margin=0.95in]{geometry}

\usepackage{
color,
amssymb,
amsthm,
graphicx,
float,
subfig,
fancyhdr,
tikz,dsfont,
caption,
amsmath}

\usepackage[colorlinks=true,linkcolor=blue, citecolor=red]{hyperref}

\usetikzlibrary{matrix}
\usetikzlibrary{positioning}
\usetikzlibrary{arrows}

\makeindex

\newtheorem{theorem}{Theorem}[section]

\newtheorem{lemma}[theorem]{Lemma}
\newtheorem{corollary}[theorem]{Corollary}

\newtheorem{definition}[theorem]{Definition}
\newtheorem{construction}[theorem]{Construction}
\newtheorem{remark}[theorem]{Remark}

\newtheorem{example}[theorem]{Example}

\title{On the complexity of zero-dimensional multiparameter persistence}

\author{Jacek Brodzki}
\email{j.brodzki@soton.ac.uk}
\author{Matthew Burfitt}
\email{m.i.burfitt@soton.ac.uk}
\author{Mariam Pirashvili}
\email{M.Pirashvili@soton.ac.uk}
\address{Department of Mathematical Sciences, University of Southampton, Highfield, Southampton, SO17 1BJ, United Kingdom}
\thanks{This research was supported by the EPSRC grant EP/N014189/1.}

\begin{document}

\maketitle

\begin{abstract}
    Multiparameter persistence is a natural extension of the well-known persistent homology, which has attracted a lot of interest. However, there are major theoretical obstacles preventing the full development of this promising theory.
    
    In this paper we consider the interesting special case of multiparameter persistence in zero dimensions which can be regarded as a form of multiparameter clustering.
     In particular,  we consider the multiparameter persistence modules of the zero-dimensional homology of filtered topological spaces when they are finitely generated. Under certain assumptions, we characterize such modules and study their decompositions.
    In particular we identify a natural class of representations that decompose and can be extended back to form zero-dimensional multiparameter persistence modules.
    
    Our study of this set of representations concludes that despite the restrictions, there are still infinitely many classes of indecomposables in this set.
\end{abstract}

%%%%%%%%%%

\section{Introduction}\label{sec:Introduction}

Persistent homology assigns numerical invariants to a topological space. In practice, the necessary input is provided by a filtered simplicial complex created from a sample of the space of interest \cite{Carlsson09, Edelsbrunner2014}. This leads to the formation of a persistence module, by assigning a vector space -- the associated homology group -- to each stage of the filtration, together with the linear transformations stemming from the inclusion maps of the filtration.

When the filtration is given by one parameter, i.e. when we have a nested sequence of subcomplexes, the resulting persistence module is a representation of the so-called $\mathcal{A}_n$ quiver. Due to a deep theorem by Gabriel in representation theory, this quiver is known to be of finite type. This means that every representation of it can be decomposed into a direct sum of a finite set of isomorphism classes of indecomposables.

This property only holds for representations of quivers whose shape is that of a Dynkin or Euclidean diagram. This is very restrictive and for most quivers the problem is known to be extremely hard.

The generalisations to multiparameter persistence modules are fraught with difficulties due to algebraic problems in the decomposition theory of such modules.
These are quiver representations over partially ordered sets with commutative restrictions. In practice, the poset considered is often a commutative $n$-dimensional grid.
In this paper we justify considering persistence modules over finite posets by showing that the algebraic problems are equivalent to $\mathbb{N}^n$ and $\mathbb{R}^n$ modules with a finite encoding.
However, even this restriction to finite bounded acyclic quiver representations, corresponding to multiparameter persistence, does not in any clear way simplify the problem \cite{Donovan1973, Nazarova1973}.
Multiparameter persistence modules also arise in theoretical works in topological data analysis. However, because the underlying quiver is of a wild type, it becomes much more challenging to derive numerical invariants. It has been shown \cite{Carlsson2009} that one cannot assign a complete discrete invariant to such modules and in dimensions greater than zero all modules can be realised as the homology of some filtration \cite{Harrington2017}.

Multiparameter persistence is strongly motivated by applications such as including density variables \cite{Wasserman2018} to extend the stability of persistent homology with respect to noise
and to incorporate other data specific parameters, the uses of which have already been demonstrated \cite{Carlsson2008, Adcock2012, keller2018}.

Studying multiparameter persistent homology in zero dimensions is equivalent to considering multiparameter clustering, which is relevant to problems concerning clustering density.
In fact, the origins of persistent homology lie in the study of the structure of connected components \cite{frosini1990} and has been applied to the stability of hierarchical clustering \cite{Carlsson2010, Carlsson2010b}.

The description of zero dimensional multiparameter persistence modules becomes somewhat more combinatorial than that of classical quiver representation theory so that we might begin to understand their decomposition theory. The reason for this is that the class of persistence modules arising  as zero-dimensional persistence is not necessarily stable 
with respect to natural representation-theoretic transformations, for example,  the reflection functors.
It therefore makes sense to understand the special case of zero-dimensional multiparameter persistence modules using more combinatorial tools.

This is the question we are investigating in this paper. Restricting our attention 
to the zero-dimensional case allows us to consider a smaller class of representations. As a result, we are able to show that in general, the isomorphism classes of indecomposable representations of the remaining modules are still infinite.

To this end, we introduce component modules corresponding to the persistence modules obtained from zero-dimensional multiparameter persistence
and the weaker notion of semi-component modules, obtained as decompositions of component modules.
We show that provided there is a minimal generator in the component module, it splits as an interval and semi-component module.
A situation that occurs in applications such as hierarchical clustering.
Furthermore any component module can be recovered from a semi-component module up to a choice of right inverse dependent only on the placement of the minimal generator.

A number of approaches to multiparameter persistence have been developed \cite{Corbet2018, Harrington2017, Lesnick2015, Miller2017, Scolamiero2017, Skryzalin2017}.
Most relevant to the content of our paper, much work has also been devoted to studying the structure of multiparameter persistence modules under restrictions.
For two-parameter persistent homology, Cochoy and Oudout \cite{Cochoy2016} present a local conditions that characterises persistence module emerging as a sum of interval modules.
However in simple cases on $\mathbb{N}^m$, Buchet and Escolar \cite{Buchet2018, Buchet2019} have given realistically plausible topological examples that realise complex infinite families of indecomposable modules, in which any one parameter persistence module can be obtained by the restriction of a two-parameter indecomposable module.
Motivated by multi-filtrations arising form hierarchical clustering Bauer, Botnan and Oppermann \cite{Bauer2019} investigated two-parameter persistence modules in which the parameter choices induces injective or surjective maps vertically or horizontally within the module, giving a full classification of representation types.

%%%%%%%%%%

\section{Background}\label{sec:Background}

	In this section we recall the language and relevant results on multiparameter persistence useful to us in subsequent sections.
	Unless otherwise stated assume we are working over a field $\mathbb{K}$ and all vector spaces over $\mathbb{K}$ are finite dimensional.

    %%%%%%%%%%
	
	\subsection{Partially ordered sets}
	
	Multiparameter persistence modules occur indexed over partially ordered sets,
	therefore the mathematics of theses structures plays a fundamental part in their description.
	We recall now notations concerning partially ordered sets used later in this work.
	
	\begin{definition}\label{def:PartialOrder}
	A \emph{partially ordered set}, a \emph{partial order} on a set or a \emph{poset} is a set $P$ with a binary relation $\leq$ that satisfies the following.
	    \begin{enumerate}
	        \item
	        For any $p \in P$, then $p \leq p$ (Reflexivity); 
	        \item
	        For any $p,q \in P$, if $p\leq q$ and $q\leq p$ then $p=q$ (Antisymmetry);
	        \item
	        For any $p,q,r \in P$, if $p\leq q$ and $q\leq r$ then $p \leq r$ (Transitivity).
	    \end{enumerate}
	\end{definition}
	
	A partial order can be seen as a category with an object for each element of $P$ and a unique morphism $\alpha \colon p \to q$ whenever $p \leq q$ for $p,q\in P$.
	A morphism of partially ordered sets coincides with a morphism between the associated categories.
	Following this section, throughout the remainder of this paper we view a partially ordered set as a category.
	
	We use the following terminology on partial orders throughout the rest of the paper.
	A poset $P$ is \emph{finite} if $P$ is finite. 
	Two elements $p,q \in P$ are said to be \emph{comparable} if either $p \leq q$ or $q \leq p$.
	A poset is \emph{connected} if any two elements can be written as the endpoints of a sequence of pairwise comparable elements.
    We say element $p \in P$ is \emph{minimal} if for any comparable $q \in P$, then $p\leq q$.
    Similarly an element $p \in P$ is said to be \emph{maximal} if for any comparable $q \in P$, then $q \leq p$.
    A poset is \emph{bounded} if it has unique minimal and maximal elements.
	For $p \leq q$ such that $p\neq q$ and for $r\in P$ such that $p\leq r \leq q$ implies $r=p$ or $r=q$,
	then we say $q$ \emph{covers} $p$ and we call $(p,q)$ a \emph{cover} relation.
	A graph with vertices the element of the partial order and edges the cover relations between them is called a \emph{Hasse diagram}. 
	
	A subset $C\subseteq P$ is called a \emph{chain} if it is \emph{totally ordered},
	that is if any two elements of $C$ are comparable.
    A chain is called a \emph{maximal chain}, if it is not contained in a larger chain.
    The \emph{height}, $h(P)$ of a partial order $P$ is the cardinality of a maximal chain in $P$.  
    An \emph{antichain} is a subset of $P$ in which no pair of elements are comparable.
    Given two partially ordered sets $P$ and $Q$ their product is the partially ordered set on $P\times Q$ such that,
    \begin{equation*}
        (p,q) \leq (p',q') \text{ if and only if }
        p \leq p' \text{ in } P \text{ and } q \leq q' \text{ in } Q.
    \end{equation*}
	
	%%%%%%%%%%
	
    \subsection{Multiparameter persistence modules}\label{sec:MultiparameterModules}
	
	In their most general form multiparameter persistence modules are considered indexed over an arbitrary partially ordered set.

	\begin{definition}\label{PersistenceModule}
		A (multiparameter) {\it persistence module} $M$ over a field $\mathbb{K}$ is a functor from a poset $P$ to the category of finite dimensional $\mathbb{K}$-vector spaces and linear maps.
		A persistence module $M$ is called finite if $P$ is finite.
	\end{definition}
	
	For simplicity from now on we call a multiparameter persistence module a persistence module.
	In the usual way, we obtain a category of persistence modules and natural transformations.
	Given a persistence module $M$ the \emph{dimension vector} of $M$ is 
	\begin{equation*}
	    \{ \dim{M(p)} \}_{p\in P},
	\end{equation*}
	the set of dimensions of vector spaces assigned to each object of the partial order $P$.
	Since persistence modules are indexed over a partial order we present them as a diagram, it makes sense to present them as Hasse diagrams, (adding direct edges).
	That is vertices are labelled with the dimension of the vector space and cover relations are labelled with the corresponding linear map with respect to a chosen basis for each vector space.
	
    The modules naturally corresponding to the barcode decomposition in one parameter persistence (persistence over a totally order set) correspond to intervals in the order.
	The notion of an interval module is generalized easily to the multiparameter setting.

	\begin{definition}\label{def:IntervalComponent}
		A $\mathbb{K}$ persistence module $M$ over a partially ordered set $P$ is called an \emph{interval module} if:
		\begin{enumerate}
			\item
			For every object $p\in P$, the vector space $M(p)$ is isomorphic to $\mathbb{K}$ or $0$.
			\item
			For all morphisms $\alpha \colon q\to p$ in $P$ such that $M(q) \cong M(p) \cong \mathbb{K}$, the linear maps $M(\alpha)$ can with respect to some basis be chosen simultaneously as the identity.
			\item
			For morphisms $\alpha \colon q\to p$ and $\beta \colon p \to l$ in $P$
			if $M(q) \cong M(l) \cong \mathbb{K}$, then $M(p) \cong \mathbb{K}$.
		\end{enumerate} 
	\end{definition} 
	
	Note that \cite{Asashiba2018} if part (2) of the definition is weakend to require only isomorphism, then the definition is weakened for modules that cannot be embedded in $\mathbb{N}^m$ for $m > 2$.  
	
	Multiparameter persistence modules are of interest because they arise as the homology of filtered spaces.
    
	\begin{definition}\label{def:FilteredSpace}
		A {\it filtered topological space} is a functor $X$ from a poset $P$ to the category of topological spaces and inclusions, such that the image of every object under $X$ has finitely many path components.
	\end{definition}

    We assume here that the filtered space $X$ is a simplicial complex filtered by subcomplexes and that the inclusions are simplicial inclusions.
    This can be justified by the fact that topological spaces arising form data, such as the Vietoris-Rips and Alpha complexes obtained form point cloud data are realised in this form.

    %%%%%%%%%%
	
	\subsection{The theory of quiver representations}\label{sec:Quiver}
	
	A finite poset can be identified with a \emph{quiver}.
	This identification is an important step in applying well-known results from the theory of quiver representations to the setting of persistence modules.
	In particular the theory of quiver representations can be applied to multiparameter persistence modules.
	Unless otherwise stated all the proofs of statements made in this section can be found in \cite{Skowronski2007}.
	
	\emph{Quivers} can be thought of as directed graphs, or rather multigraphs. In general, both multiple arrows between nodes and loops (i.e. arrows beginning and ending at the same node) are allowed. However, in the case of persistence modules, we only deal with quivers that are simple directed graphs.
	A \emph{representation} of a quiver $Q$ over a field $\mathbb{K}$ is the association of a $\mathbb{K}$-vector space to each node of $Q$ and a linear map between the vector spaces to each arrow.
	If each vector space is finite dimensional then the representation is called \emph{finite dimensional}.
	The direct sum of two quiver representations is the quiver representations whose vector spaces and linear maps at are the direct sums of those form the two representations.

	\begin{definition}\label{def:IndecomposableModule}
	    A representation of a quiver $Q$ is called \emph{indecomposable} if it cannot be expressed as a direct sum of at least two non-zero representations of $Q$.
	\end{definition}

    There is an equivalence of categories between a quivers representations and left modules over their path algebra.
    It is then a straightforward consequence of the classical Krull-Remak-Schmidt theorem, that a decomposition of a finite-dimensional quiver representation into a direct sum of indecomposables is unique up to reordering and isomorphism.

	The following lemma appears as the most straightforward characterisation of indecomposable quiver representations, see \cite{Brion2012}.
	However here for ease of use, we state the lemma in terms of persistence modules.
	
	\begin{lemma}\label{lem:DecompositionCharectersiation}
		A persistence module is indecomposable if and only if all self natural transformations are either an isomorphism or nilpotent. 
		In particular, given a natural transformation $f$ of persistence module $M$ we have
		\begin{equation*}
			M= \text{ker}(f^n) \oplus \text{im}(f^n),
		\end{equation*}		
		for $n$ large enough and where $f^n$ is the composition of $f$ with itself $n$ times.
	\end{lemma}

    It is straightforward for example to show using Lemma \ref{lem:DecompositionCharectersiation} 
	that all persistence modules $M$ over $P$ satisfy the first two conditions of Definition \ref{def:IntervalComponent} are indecomposable persistence modules, since for $q\in Q$ if in $M(q)$ is non-trivial then the choice of a natural endomorphism $f$ at $M(q)$ determines the choice at all other $M(p)$ by naturality, hence $f$ is either a natural isomorphism or trivial.

%%%%%%%%%%

\section{Finite encoding}\label{sec:Encoding}

    In this section we justify the abstract setting for multiparameter persistence modules used in this work.
    The term \emph{finite encoding} is used by Miller \cite{Miller2017}, where given a persistence module $M$ over $Q$ an encoding of $M$ by a poset $P$ is a poset morphism 
    \begin{equation*}
        \pi\colon Q \to P,
    \end{equation*}
    together with a $P$-module $H$ such that $M$ is the pullback of $H$ along $\pi$, which is naturally a $Q$-module. The encoding is finite if
    \begin{enumerate}
        \item $P$ is finite, and
        \item the vector space $H_p$ is finite dimensional for all $p\in P$.
    \end{enumerate}

    The next theorem shows that in terms of decomposition, considering finite multiparameter persistence modules is equivalent to considering $\mathbb{N}^n$ and $\mathbb{R}^n$ modules with a finite encoding.
    To prove the theorem we first give the following additional terminology.
     Given $p_1,\dots,p_m$ in poset $P$, we call the set $J$ the \emph{join-set} of $p_1,\dots,p_n$ if it satisfies the following conditions.
    \begin{enumerate}
        \item
        For each $j\in J$ we have $j\geq p_i$ for $i=1,\dots,m$,
        \item
        any $k\in P$ satisfying the first condition is either incomparable to each  $j\in J$, $j\neq k$,  or $j \leq k$.
    \end{enumerate}
    If the join set contains a single element then this is the usual notion of a join.
    
    \begin{theorem}\label{thm:EncodingEmbedding}
        Every finite persistence module with a single maximal and minimal element in its poset can be realised as the finite encoding of some $\mathbb{N}^n$ module, and by extension, a right continuous $\mathbb{R}^n$ module constant on cubical regions.
    \end{theorem}
    
    We can extend the theorem to all finite multi-parameter modules if we identify them with the module obtained by attaching a zero-dimensional vector space as the minimal and or maximal elements of the module.
    
    \begin{proof}
        Let $M$ be a finite multiparameter persistence module over poset $P$
        and let
        \begin{equation*}
            C_1,C_2\dots,C_n
        \end{equation*}
        be a cover of $P$ by maximal chains.
        We construct a poset morphism $\pi \colon \mathbb{N}^n \to P$ such that its pullback over $M$ is a well defined $\mathbb{N}^n$ module for which $\pi$ is a finite encoding.
        Let $l_i$, for $i=1, \dots n$,  be the length of the maximal chain $C_i$. Select the $l_i$  smallest elements of axis $i$  and assign them bijectively to 
        the elements of $C_i$ in an order-preserving way.
        As the chains are maximal, $\pi$ is well defined at the origin sending it to the minimal element of $P$.
        
        We now extend $\pi$ working inductively on elements of $\mathbb{N}^n$ using an analogue of Cantor's diagonal argument on maximal antichains in $\mathbb{N}^n$ with respect to their natural order. 
        Let $y\in \mathbb{N}^n$ on which $\pi$ has not yet been defined, and let 
        $U_y\subset \mathbb{N}^n$ be the set of elements $m\leq y$. We let $\pi(y)$ be an element of the join-set of the subset $\pi(U_y)\subseteq P$. 
        Since we assume a single maximal element of $P$, the join-set of any set of elements is always non-empty.
        
        By construction, for every $y\in \mathbb{N}^n$, $\pi(x) \leq \pi(y) $ for every 
        $x\in\mathbb{N}^n$ such that $x\leq y$ and so $\pi$ is a poset map as required. 
        Hence $\pi$ is a well defined encoding of its pullback modules and it is also finite as $P$ is finite.
    \end{proof}

\begin{remark}
    There are classical theoretical results in algebraic topology that justify Miller's definition of finite encoding as a sensible approach to multiparameter persistence.
    
    A function that allows finite encoding is one that has constant homology on the preimages of compact regions of the range, i.e that has regions of critical values that form submanifolds of the range of one lesser dimension.
    In particular, finite encoding restricts to the usual finiteness assumption present for one-parameter persistence. Therefore, the setting of modules over an arbitrary finite partially ordered set considered here is of particular interest.
    We mention here two results from topology, which apply to the case of smooth functions
    \begin{equation*}
        f_0,f_1,\dots,f_n\colon \mathcal{M} \to \mathbb{R}
    \end{equation*}
    on a smooth $d$ dimensional manifold $\mathcal{M}$.
    Sard's Theorem \cite{sard} states that in the case when $\mathcal{M} = \mathbb{R}^m$,
    while there might be many critical points in $\mathbb{R}^m$, the image, i.e. the set of the critical values, has Lebesgue measure $0$ in $\mathbb{R}^n$. This is of course more general than the set of critical values being a submanifold.
    
    The \emph{Jacobi set} of $f_0,f_1\dots,f_n$ is the closure of
    \begin{equation*}
        \{ x\in \mathcal{M} \; | \; x \text{ is a critical point of } f_i \text{ for some } i=0,1 \dots ,n  \}.
    \end{equation*}
    Edelsbrunner and Harer \cite{Edelsbrunner2004} worked on the structure of such critical sets for the applied setting.
    In particular they show that when $n < 4$ the Jacobi set is a smoothly embedded sub-manifold of $\mathcal{M}$.
\end{remark}

%%%%%%%%

\section{Topological restrictions on zero-degree persistence modules}\label{sec:Topology}

It is stated \cite{Carlsson2009} and proved in \cite{Harrington2017}, that when $P=\mathbb{N}^n$  every finitely generated multiparameter persistence module over $P$ with coefficients in $\mathbb{F}_p$ or $\mathbb{Q}$ can be realised as  homology in positive degrees of a filtered simplicial complex. 
Using Theorem \ref{thm:EncodingEmbedding}, it is clear that these results will also hold over an arbitrary finite partially ordered set. 
However, for $i=0$, we see that the situation is more restricted.

Let $\pi_0(X)$ be the set of (path-)connected components of the simplicial complex $X$.
It is an  elementary fact in topology that the image of a continuous map on a (path-)connect component is (path-)connected.
So $\pi_0$ is a functor from the category of simplicial complexes and continuous maps to the category of sets.
It is a classical fact from algebraic topology \cite{Hatcher2002} that $H_0(X)$ is the free vector space generated by $\pi_0(X)$. Therefore, when we talk about a basis for $H_0(X)$, we always talk about this particular basis (unless stated otherwise). It follows that if $f:X\to Y$ is a continuous map, then the induced map $H_0(X)\to H_0(Y)$ has the property that it takes basis elements to basis elements. Therefore, the matrix representing this map consists of zeros and ones, with exactly one $1$ in each column.

\subsection{Factoring through \emph{Sets}}\label{sec:Sets}
The above implies that the induced maps between the homology groups are factoring through the category of sets. This property is rather fragile, however. Standard transformations in representation theory, like reflection functors, destroy it. Take, for example, the following representation of $A_2$:
 \begin{equation*}
    \begin{tikzpicture}[>=triangle 45]
        \node (A) at (-1,0) {$\mathbb{K}^2$};
        \node (B) at (1,0) {$\mathbb{K}$.};
        \draw [semithick,->] (A) -- (B) node [midway,below] 
            {$\begin{bmatrix} 1 & 1 \end{bmatrix}$};
    \end{tikzpicture}
\end{equation*}
The reflection functor reverses the arrow, and takes this to the representation
\begin{equation*}
    \begin{tikzpicture}[>=triangle 45]
        \node (A) at (-1,0) {$\mathbb{K}^2$};
        \node (B) at (1,0) {$\mathbb{K}$.};
        \draw [semithick,->] (B) -- (A) node [midway,below] 
            {$\begin{bmatrix} -1 \\ 1 \end{bmatrix}$};
    \end{tikzpicture}
\end{equation*}
Clearly, this representation no longer satisfies our condition, as the corresponding matrix does not consist solely of $0s$ and $1s$. In this particular case, we are dealing with a representation of the $A_2$ quiver, and therefore there is a nice change of basis that makes it fit our condition (and, in fact, both of these representations decompose). However, this is no longer necessarily the case for more complicated quivers.

Geometrically, the reflection functor takes subspace embeddings to quotient maps, and vice versa. In this case, for example, the second is embedding the subspace $\mathbb{K}$ along the second diagonal in $\mathbb{K}^2$, while the first representation corresponded to $\mathbb{K}^2$ being projected onto the quotient space $\frac{\mathbb{K}^2}{\mathbb{K}}$. So we have a short exact sequence
\begin{equation*}
    \begin{tikzpicture}[>=triangle 45]
        \node (A) at (-2,0) {$\mathbb{K}$};
        \node (B) at (0,0) {$\mathbb{K}^2$};
        \node (C) at (2,0) {$\mathbb{K}$,};
        \draw [semithick,->] (A) -- (B) node [midway,below] 
            {$\begin{bmatrix} -1 \\ 1 \end{bmatrix}$};
        \draw [semithick,->] (B) -- (C) node [midway,below] 
            {$\begin{bmatrix} 1 & 1 \end{bmatrix}$};
    \end{tikzpicture}
\end{equation*}
and the composition of the two matrices gives the zero map.
	
	%%%%%%%%%%
	
	\subsection{Component and semi-component modules}\label{Sec:Component}
	
	We first define our main object of study motivated by the discussion at the beginning of the section.

	\begin{definition}\label{def:ComponentModule}
	    A $\mathbb{K}$ persistence module $M$ over a partially ordered set $P$ is called a \emph{component module} if there is a basis (called a \emph{component basis} of $M$) for every vector space $M(q)$, with $q \in P$,
	    such that any morphism $\alpha \colon q\to p$ in $P$ with respect to the basis is a linear map whose matrix contains only ones and zeros with exactly one $1$ in each column. If the matrices contain at most one $1$ in each column, it is called a \emph{semi-component module}.
	\end{definition}

    Interval modules are the component modules for which $M(q)$ has at most dimension one for any $q\in P$.
	For any multiparameter persistence module we may define its generators. Here we give a simple characterisation.

	\begin{definition}\label{def:ComponetGenerator}
	    An element of the component basis is called a generator if it is not in the image of any of the linear maps $M(\alpha)$ for $\alpha \colon p\to q$ in $P$.
	    A (semi-)component modules is said to have a minimal generator if, 
	    there is at least one generator in $M(p)$ and maps
		\begin{equation*}
		    \alpha \colon p \to q,
		\end{equation*}
		for any $q\in P$ such that $M(q)$ contains a generator.
	\end{definition}

	The set of generators can be chosen to interact well with component basis.

	\begin{lemma}\label{lem:ComponentGenerator}
		The elements of a (semi-)component basis of a (semi-)component module can be chosen so that each basis element is the image of one or more generators among all morphisms for which it is in the domain.  
	\end{lemma}
	
	\begin{proof}
		The lemma follows up to changes of sign from the structure of the matrices with respect to a component basis
		given in definition \ref{def:ComponetGenerator}.
	\end{proof}
    
    The next theorem demonstrates that component modules are the important modules for zero dimensional persistent homology.
	In particular, the component modules characterize exactly the modules obtainable as the zero-dimensional homology of filtered topological spaces.

	\begin{theorem}\label{thm:ComponentAlgebra}
	    The zero dimensional homology of a finite filtered CW-complex over partial order $P$ is a finitely generated component module over $P$.
		In addition, every finitely generated component module over $P$ is the zero-dimensional homology of a some finite filtered CW-complex over $P$.
	\end{theorem}

	\begin{proof}
		For a finite filtered CW-complex $X$ over a partially ordered set $P$, we may choose a basis of each $H_0(X(q);\mathbb{K})$ as the set of homology classes of points in each path component of $X(q)$ for each $q \in P$.
		These basis are each finite since $X$ is finite.  
		The morphisms of the persistence module $H_0(X;\mathbb{K})$ are induced by the the inclusions $X(\alpha)$ for $\alpha:q\to p$ a in $P$.
		By continuity the image of a path component of $X(q)$ under $X(\alpha)$ is contained in a single path component of $X(p)$,
		which implies the conditions of definition \ref{def:ComponentModule}.
		By construction the number, generators of $H_0(X;\mathbb{K})$ is bounded above by the number of path components in $X$, which can be no more than the number of zero simplicis and
		since each $X(q)$ is a finite CW-complex $H_0(X;\mathbb{K})$ is finitely generated.
		
		Given a finitely generated component module $M$ over a partially ordered set $P$ we may construct a filtered simplicial complex $X$ such that $M=H_0(X;k)$ as follows.
		For each generator $g_i \in M(q_i)$ of $M$, there is a corresponding vertex $x_i^q$ in each $X(q)$ for $q\in P$ with a morphism $\alpha\colon q_i \to q$.
		In particular by Lemma \ref{lem:ComponentGenerator}, each $x_i^q$ correspond to one or more $v_j^q$ in the component basis. 		
		For every morphism $\alpha \colon q \to p$ in $P$
		there is an edge between $x_i^q$ and $x_j^q$ when both have a $1$
		in the same row of the matrix corresponding to $\alpha$ with respect to the component basis.
		The edges can be chosen consistently because all morphism in $M$ commute.
		Finally, since there were only finitely many generators of $M$, by construction each $X(q)$ is a finite CW-complex.
	\end{proof}

    As a consequence of the characterisation of Theorem \ref{thm:ComponentAlgebra}, we can specify the subcategory of persistence modules of interest when restricting to zero dimensional homology. 
    
    \begin{definition}
        Over any a partial order $P$, define the category of \emph{zero-parameter modules} to be the full subcategory of persistence modules, 
        with objects the modules whose indecomposables occurs as indecomposables of some component module.
    \end{definition}
    
    Theorem \ref{thm:ComponentAlgebra} can be seen as following intuitively form properties of continuous maps between path components of topological spaces, however it immediately leads to the following surprising consequence on the types of decomposition that can arise form the $0$-dimensional persistent homology of filtered finite dimensional CW-complexes.
    This marks a significant simplification over the general case of persistence modules.

    \begin{corollary}
        Over any finite partial order set $P$,
        the number of zero-parameter modules of a given dimension vector is finite. 
    \end{corollary}
    
    \begin{proof}
        By Theorem \ref{thm:ComponentAlgebra} the zero-dimensional persistent homology of a filtered finite simplicial complex is a component module.
        Every component module of a fixed dimension vector over $P$ has a basis on each vector space, for which the linear maps of each morphisums are matrices containing only zeros and ones.
        Therefore there are at most finitely many such modules of a given dimension vector.
    \end{proof}

%%%%%%%%%%%%%%%%%%%%%%%%%%%%%%%%%%%%%%%%%%%%%%%%%%%%%%%%%%%%%%

\section{Decompositions of zero dimensional persistence modules}\label{sec:Decompositions}

	In this section we explore the indecomposables and decompositions of persistence modules corresponding to the zero-dimensional homology of filtered topological spaces.

    %%%%%%%%%%
	
	\subsection{An initial interval decomposition}\label{sec:IntervalDecomposition}
	
	Not all non-interval component modules allow interval decomposition, and those that do, do not necessarily decompose into interval modules as demonstrated by the following examples.  
	
	\begin{example}\label{exam:IncomparableComponent}
	    Consider the persistence module   
	    \begin{equation*}\label{eq:Counter1}
    	    \begin{tikzpicture}[>=triangle 45]
                \node (A) at (0,0) {$\mathbb{K}$};
                \node (B) at (0,2) {$\mathbb{K}$};
                \node (C) at (2,1) {$\mathbb{K}^2$};
                \node (D) at (4,1) {$\mathbb{K}$,};
                \draw [semithick,->] (A) -- (C) node [midway,below] 
                {$\begin{bmatrix} 0 \\ 1 \end{bmatrix}$};
                \draw [semithick,->] (B) -- (C) node [midway,above]
                {$\begin{bmatrix} 1 \\ 0 \end{bmatrix}$};
                \draw [semithick,->] (C) -- (D) node [midway,above]
                {$\begin{bmatrix} 1 & 1 \end{bmatrix}$};
            \end{tikzpicture}
	    \end{equation*}
	    where the matrices represent the linear maps with respect to some component basis.
	    The module is not an interval module as the central vector space is $2$-dimensional. 
	    Moreover, it is straightforward to show that the module is indecomposable using Lemma \ref{lem:DecompositionCharectersiation}.
	    Notice also that there is no minimal generator since
	    the two leftmost vector spaces lie in incomparable minimal positions of the partial order.
	    \end{example}
	    
	    \begin{example}\label{exam:ComponentDecompositions}
	    Consider the component module with component basis,  
	    \begin{equation*}
    	    \begin{tikzpicture}[>=triangle 45]
    	        \node (X) at (0,1) {$\mathbb{K}$};
                \node (A) at (3,0) {$\mathbb{K}^2$};
                \node (B) at (3,2) {$\mathbb{K}^2$};
                \node (C) at (6,1) {$\mathbb{K}^3$};
                \node (D) at (9,1) {$\mathbb{K}^2$};
                \node (Z) at (12,1) {$\mathbb{K}$.};
                \draw [semithick,->] (X) -- (A) node [midway,below] 
                {$\begin{bmatrix} 0 \\ 1 \end{bmatrix}$};
                \draw [semithick,->] (X) -- (B) node [midway,above]
                {$\begin{bmatrix} 1 \\ 0 \end{bmatrix}$};
                \draw [semithick,->] (A) -- (C) node [midway,below] 
                {$\begin{bmatrix} 0 & 1 \\ 0 & 0 \\ 1 & 0 \end{bmatrix}$};
                \draw [semithick,->] (B) -- (C) node [midway,above]
                {$\begin{bmatrix} 1 & 0 \\ 0 & 1 \\ 0 & 0 \end{bmatrix}$};
                \draw [semithick,->] (C) -- (D) node [midway,above]
                {$\begin{bmatrix} 1 & 0 & 0 \\ 0 & 1 & 1 \end{bmatrix}$};
                \draw [semithick,->] (D) -- (Z) node [midway,above]
                {$\begin{bmatrix} 1 & 1 \end{bmatrix}$};
            \end{tikzpicture}
	    \end{equation*}
	    We may construct a decomposition of this module using using Lemma \ref{lem:DecompositionCharectersiation}
	    with idempotent endomorphism,
	    \begin{equation*}
	        \begin{tikzpicture}[>=triangle 45]
    	        \node (X) at (0,1) {$\begin{bmatrix} 0 \end{bmatrix}$};
                \node (A) at (3,0) {$\begin{bmatrix} 1 & 0 \\ -1 & 0 \end{bmatrix}$};
                \node (B) at (3,2) {$\begin{bmatrix} 0 & -1 \\ 0 & 1 \end{bmatrix}$};
                \node (C) at (6,1) {$\begin{bmatrix} -1 & -1 \\ 1 & 0 \\ 0 & 1 \end{bmatrix}$};
                \node (D) at (9,1) {$\begin{bmatrix} 0 & -1 \\ 0 & 1 \end{bmatrix}$};
                \node (Z) at (12,1) {$\begin{bmatrix} 0 \end{bmatrix}.$};
                \draw [semithick,->] (X) -- (A) node [midway,below]{};
                \draw [semithick,->] (X) -- (B) node [midway,above]{};
                \draw [semithick,->] (A) -- (C) node [midway,below]{};
                \draw [semithick,->] (B) -- (C) node [midway,above]{};
                \draw [semithick,->] (C) -- (D) node [midway,above]{};
                \draw [semithick,->] (D) -- (Z) node [midway,above]{};
            \end{tikzpicture}
	    \end{equation*}
	    This yields a decomposition into kernel and image modules
	    \begin{equation*}
    	    \begin{tikzpicture}[>=triangle 45]
    	        \node (X) at (0,1) {$\mathbb{K}$};
                \node (A) at (3,0) {$\mathbb{K}$};
                \node (B) at (3,2) {$\mathbb{K}$};
                \node (C) at (6,1) {$\mathbb{K}$};
                \node (D) at (9,1) {$\mathbb{K}$};
                \node (Z) at (12,1) {$\mathbb{K}$};
                \draw [semithick,->] (X) -- (A) node [midway,below] 
                {$\begin{bmatrix} 1 \end{bmatrix}$};
                \draw [semithick,->] (X) -- (B) node [midway,above]
                {$\begin{bmatrix} 1 \end{bmatrix}$};
                \draw [semithick,->] (A) -- (C) node [midway,below] 
                {$\begin{bmatrix} 1 \end{bmatrix}$};
                \draw [semithick,->] (B) -- (C) node [midway,above]
                {$\begin{bmatrix} 1 \end{bmatrix}$};
                \draw [semithick,->] (C) -- (D) node [midway,above]
                {$\begin{bmatrix} 1 \end{bmatrix}$};
                \draw [semithick,->] (D) -- (Z) node [midway,above]
                {$\begin{bmatrix} 1 \end{bmatrix}$};
            \end{tikzpicture}
	    \end{equation*}
	    and
	    \begin{equation*}
    	    \begin{tikzpicture}[>=triangle 45]
    	        \node (X) at (0,1) {$0$};
                \node (A) at (3,0) {$\mathbb{K}$};
                \node (B) at (3,2) {$\mathbb{K}$};
                \node (C) at (6,1) {$\mathbb{K}^2$};
                \node (D) at (9,1) {$\mathbb{K}$};
                \node (Z) at (12,1) {$0$.};
                \draw [semithick,->] (X) -- (A) node [midway,below] 
                {$\begin{bmatrix} 0 \end{bmatrix}$};
                \draw [semithick,->] (X) -- (B) node [midway,above]
                {$\begin{bmatrix} 0 \end{bmatrix}$};
                \draw [semithick,->] (A) -- (C) node [midway,below] 
                {$\begin{bmatrix} 0 \\ 1 \end{bmatrix}$};
                \draw [semithick,->] (B) -- (C) node [midway,above]
                {$\begin{bmatrix} 1 \\ 0 \end{bmatrix}$};
                \draw [semithick,->] (C) -- (D) node [midway,above]
                {$\begin{bmatrix} 1 & 1 \end{bmatrix}$};
                \draw [semithick,->] (D) -- (Z) node [midway,above]
                {$\begin{bmatrix} 0 \end{bmatrix}$};
            \end{tikzpicture}
	    \end{equation*}
	    Both of these are indecomposable component modules, only the first of which is an interval module. The second module can be shown to be indecomposable in the same way as Example \ref{exam:ComponentDecompositions}. 
	\end{example}

    We will see that if a non-interval component module has a minimal generator, it always decomposes.
    The following theorem specialises in the one parameter case  to the remark that zero-dimensional persistent homology decomposes with a unique longest interval, providing the filtration spaces are eventually connected.

	\begin{theorem}\label{thm:ComponentSplit}
		Let $C$ be a finite $\mathbb{K}$ component module over a partially ordered set $P$ with at least two generators, at least one of which is minimal.
		Then $C$ decomposes as a direct sum of an interval and a non-trivial semi-component module over $P$.
	\end{theorem}

	\begin{proof}
		By Lemma \ref{lem:ComponentGenerator}
		we can take generators 
		\begin{equation*}
			g_1,\dots, g_n \text{ in } C(q_1),\dots,C(q_n),
		\end{equation*}
		with $g_i\in M(q_i)$ and component basis
		\begin{equation*}
			v_1^q,\dots ,v_{\dim(C(q))}^q\in C(q),
		\end{equation*}
		such that for each $p\in P$,
		any $P$ morphism $\alpha\colon q_i \to p$ has $C(\alpha)(g_i)=v_j^p$ for each $i=1,\dots,n$ and $j=1,\dots, \dim(C(p))$.
		Let $d\in \mathbb{N}$ be the maximal dimension of $C(p)$ over all $p\in P$.
		
		If $d=0$ then $M(p)=0$ for all $p\in P$, so $C$ is the trivial module.
		If $d=1$ then $M$ is an interval module.
		Therefore, assume now that $d\geq 2$ and as a consequence $n \geq 2$ also.
		We construct a natural endomorphism $f$ on $C$ that is neither an isomorphism nor nilpotent.
		Since $C$ has a minimal generator, without loss of generality we may assume that $g_1$ is a minimal generator.
		So, for each $j=1,\dots,n$ there  are morphisms
		\begin{equation*}
			\alpha_j \colon q_1 \to q_j
		\end{equation*}
		in $P$.
		Using the commutative condition, a natural endomorphism $f$ of $C$ will be determined uniquely by $f_{q_i}(g_i)$ for $i= 1,\dots,n$.
		Define $f_{q_i}$ on $g_i$ by   		
		\begin{equation}\label{eq:fatq}
			f_{q_i}(g_i) =
			\begin{cases}
    				g_i-M(\alpha_i)(g_1)		&		\text{if } i \neq 1 \\
    				0							&		\text{if } i=1.
    			\end{cases}
		\end{equation}
		We now check that this extends to a natural endomorphism $f$ on $C$.
		For any $p\in P$ and $v_k^p$ for $k=1,\dots,\dim(C(p)))$, let $\{ i_1,\dots,i_l \}$ be the subset of $\{1,\dots,n\}$ such that
		there exists a $\beta_{i_j} \colon q_{i_j} \to p$ with
		\begin{equation*}
			\beta_{i_j}(g_{i_j})=v_k^p.
		\end{equation*}
		For each $i=1,\dots,n$ since $C$ is a component module if $\beta_{i}$ exists, $\beta_{i}(g_{i})$ cannot be zero and must by construction be an elements of the component basis.
		The endomorphism $f$ exists if and only if for all $i_j$, each
		\begin{equation*}
			(\beta_{i_j}\circ f_{q_{i_j}} )(g_{i_j})
		\end{equation*}
		agree, in which case $f_{p}(v^p_{k})$ is the image.
		This is indeed the case because if for some $j$, $i_j=1$ then all the $(\beta_{i_j}\circ f_{q_{i_j}} )(g_{i_j})$ are $0$,
		otherwise $(\beta_{i_j}\circ f_{q_{i_j}} )(g_{i_j})$ are all equal to $v^p_{k}-C(\alpha_1)(g_1)$. 
		Since $n \geq 2$, by construction (\ref{eq:fatq}) $f$ is not an isomorphism and not nilpotent.
		
		In fact $f$ is an idempotent, meaning $f^2=f$.
		To see this notice that for each $i=1,\dots,n$, the linear map
		$f_{q_i}$ given in equation (\ref{eq:fatq}) is an idempotent, since for $i\neq 1$ we have $f_{q_i}(C(\alpha_i)(g_1))=0$ by naturality of $f$, so
		\begin{equation*}
		    f_{q_i}^2(g_i)=f_{q_i}(g_i-C(\alpha_i)(g_1))=f_{q_i}(g_i)
		\end{equation*}
		and $f_{q_1}^2=0$ otherwise.
		By naturality $f_q$ is determined by $f_{q_i}$ for any other $q\in P$ and so $f$ is an idempotent.
		
		Hence by Lemma \ref{lem:DecompositionCharectersiation} since $f$ is an idempotent, $C$ can be decomposed as the direct sum of the kernel of $f$
		and the image of $f$.
		The kernel of $f$ is the sub module generated by the single generator $g_1$, hence is an interval module.
		The image of $f$ is generated by 
		\begin{equation*}
			g_i-g_1
		\end{equation*}
		for $i=2,\dots,n$.
		Since $C$ is a component module the restriction of the matrices with respect to the component basis on these generators is a
		a semi-component module.
		Therefore, the image of $f$ is a semi-component module with $n-1$ generators.
	\end{proof}
	
	The theorem does not apply to semi-component modules essentially because connected components in a filtered topological space can only merge and cannot disappear.
	This property is not shared by semi-component modules. 
    
    %%%%%%%%%%
    
    \subsection{Semi-component extension}\label{sec:Extensions}
	
	In this section we show that all semi-component modules may be obtained as the decomposition of a component module. In particular, we observe how this decomposition relates to the one derived in the previous section. 
	We first demonstrate the key ideas in an example. 
	
	\begin{example}\label{exam:SemiExtension}
        The component module on the right can be seen to decompose into the sum of the semi-component module shown on the left together with  an interval module given by the entire underlying poset. 
        This is done by straightforward linear transformations to diagonalise each matrix.
	    
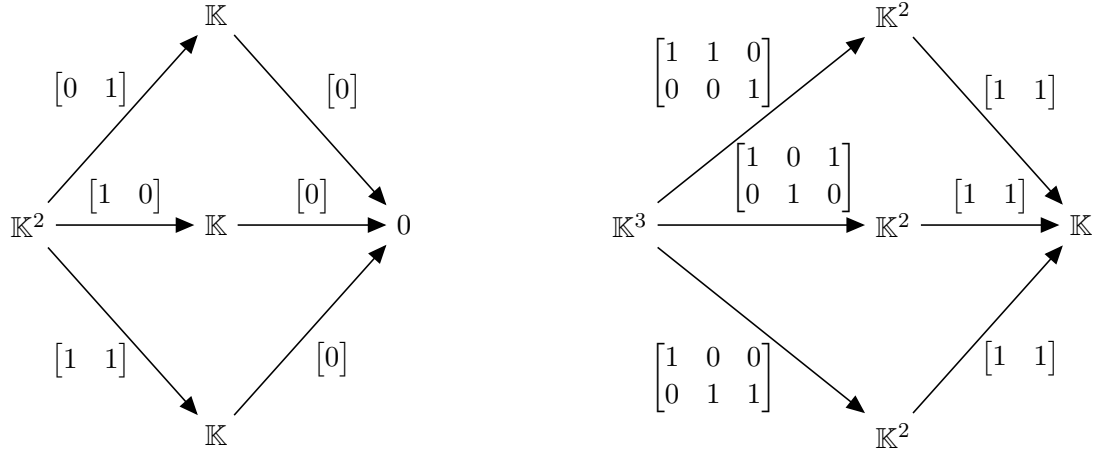
\begin{equation*}
    	    \begin{tikzpicture}[>=triangle 45]
    	        \node (X) at (0,2.5) {$\mathbb{K}^2$};
                \node (A) at (2.5,-0.3) {$\mathbb{K}$};
                \node (B) at (2.5,2.5) {$\mathbb{K}$};
                \node (C) at (2.5,5.3) {$\mathbb{K}$};
                \node (Y) at (5,2.5) {$0$};
                \node (X') at (8,2.5) {$\mathbb{K}^3$};
                \node (A') at (11.5,-0.3) {$\mathbb{K}^2$};
                \node (B') at (11.5,2.5) {$\mathbb{K}^2$};
                \node (C') at (11.5,5.3) {$\mathbb{K}^2$};
                \node (Y') at (14,2.5) {$\mathbb{K}$};
                \draw [semithick,->] (X) -- (A) node [midway,below] 
                {$\begin{bmatrix} 1 & 1 \end{bmatrix} \;\;\;\;\;\;\;\;$};
                \draw [semithick,->] (X) -- (B) node [midway,above]
                {$\begin{bmatrix} 1 & 0 \end{bmatrix}$};
                \draw [semithick,->] (X) -- (C) node [midway,above] 
                {$\begin{bmatrix} 0 & 1 \end{bmatrix} \;\;\;\;\;\;\;\;$};
                \draw [semithick,->] (A) -- (Y) node [midway,below]
                {\;\;\;\; $\begin{bmatrix} 0 \end{bmatrix}$};
                \draw [semithick,->] (B) -- (Y) node [midway,above]
                {$\begin{bmatrix} 0 \end{bmatrix}$};
                \draw [semithick,->] (C) -- (Y) node [midway,above]
                {$\;\;\;\;\;\;\;\; \begin{bmatrix} 0 \end{bmatrix}$};
                \draw [semithick,->] (X') -- (A') node [midway,below] 
                {$\begin{bmatrix} 1 & 0 & 0 \\ 0 & 1 & 1 \end{bmatrix} \;\;\;\;\;\;\;\;\;\;\;\;$};
                \draw [semithick,->] (X') -- (B') node [midway,above]
                {$\;\;\;\;\;\;\;\; \begin{bmatrix} 1 & 0 & 1 \\ 0 & 1 & 0 \end{bmatrix}$};
                \draw [semithick,->] (X') -- (C') node [midway,above] 
                {$\begin{bmatrix} 1 & 1 & 0 \\ 0 & 0 & 1 \end{bmatrix} \;\;\;\;\;\;\;\;\;\;\;\;$};
                \draw [semithick,->] (A') -- (Y') node [midway,below]
                {$\;\;\;\;\;\;\;\; \begin{bmatrix} 1 & 1 \end{bmatrix}$};
                \draw [semithick,->] (B') -- (Y') node [midway,above]
                {$\begin{bmatrix} 1 & 1 \end{bmatrix}$};
                \draw [semithick,->] (C') -- (Y') node [midway,above]
                {$\;\;\;\;\;\;\;\; \begin{bmatrix} 1 & 1 \end{bmatrix}$};
            \end{tikzpicture}
	    \end{equation*}
	\end{example}
	
	We  can generalise the construction of this example on semi-component modules to the following method. 

    \begin{construction}\label{def:SemiExtension}
        Take a semi-component module $S$ and a choice of semi-component basis.
        Since with respect to this basis every linear map is represented by a matrix with at most one $1$ in each column we may obtain form $S$ a component module $C$ as follows.
        \begin{itemize}
            \item 
            Add an extra dimension to every vector space in $S$.
            \item
            Add to each matrix in $S$ a new first column and first row.
            \item
            The new first column contains a $1$ as its first entry and zeros otherwise.
            \item
            After the first entry, the new first row has entries $0$ if the column already contains a $1$ and is $1$ otherwise.
        \end{itemize}
        We call $C$ the \emph{component extension} of $S$.
    \end{construction}

    We will see that a component module resulting from a semi-component extension is the direct sum of the extended semi-component module and an interval module.
    In particular all semi-component modules can be obtained from component modules by decomposition. First we make the following observations.

    By choosing a basis of the vector space of each vertex of a quiver $Q$, the entries of the matrices associated to the arrows give a point in the \emph{representation space} of quiver representations of a given dimension vector.
    The action of an elementary matrix operations at a vertex on the representation space is easily determined.
	
	\begin{remark}\label{rmk:SymMarixDiag}
	    Given a quiver representation, we may choose a basis for each $\mathbb{K}$-vector space at a node, realising the linear map associated to each arrow between nodes as a matrix of the appropriate dimensions.
	    A decomposition of a quiver representation corresponds to a choice of basis at each node such that the matrices of each arrow are simultaneously block matrices of the same type. 
	    A change of basis at a node corresponds to a square matrix $B$.
	    The the change of basis at this node effects each matrix $T$ of an arrow into the node, by replacing it with $TB$ and each matrix $S$ of an arrow out of the node, by replacing it with $B^{-1}S$.
	    In particular:
	    \begin{enumerate}
	    \item
	        Performing an elementary row operation on the matrix of some arrow, simultaneously performs the reverse of the elementary column operation of the same type on all matrices of arrows with the same target.
	    \item
	        Performing an elementary column operation on the matrix of some arrow,
	        simultaneously performs the reverse of the elementary row operation of the same type on all matrices of arrows with the same source.
	    \end{enumerate}
	\end{remark}

    To obtain our theorem on semi-component extension we first prove the following lemma
	and to this end we give some additional terminology.
	For each $n\geq 1$ we denote by $[n]$, the set $\{0,1,2,\dots,n \}$ with the usual total order.
	
	\begin{definition}\label{def:Grading}
	    Given a finite partially ordered set $P$, a map $f \colon P \to [h(P)-1]$ is called a \emph{pre-grading} of $P$ if $f$ is a poset map with the additional property that if $p,q \in P$ when $p < q$ then $f(p) < f(q)$.
	\end{definition}
	
	Usually $P$ is called a \emph{graded} poset if there is a map $f \colon P \to \mathbb{N}$,
	that in addition to the property above also satisfies,
	$f(p)=f(q)+1$ whenever $p$ covers $q$. 
	Not all posets have a grading.
	For example a bounded poset has a grading only when all maximal chains have the same size. 
	
	\begin{lemma}\label{lem:Grading}
	    Any partially ordered set with finite height has a pre-grading.
	    In particular the preimage of each element of the pre-grading function will be an antichain.
	\end{lemma}
	
	\begin{proof}
	    Given a finite partial order $P$, we can construct a pre-grading function $f$ by induction on $[h(P)-1]$ as follows.
	    Define the preimage of $f$ at $0$ as all minimal elements of $P$.
	    Minimal elements form an antichain as they are incomparable by definition,
	    making $f$ well defined on these elements of $P$.
	    
	    Assume inductively that the preimage of $f$ has been defined on all elements with preimage less than $k\in [h(P)-1]$.
	    Let $S\subseteq P$ be all elements which cover $f^{-1}(k-1)$.
	    Take $S'\subseteq S$ to be the subset such that for any $p,q\in S$, if $p \leq q$ then $q \notin S'$.
	    Hence $S'$ is an antichain and we may choose the preimage of $f$ at $k$ to be $S'$.
	    By construction of $f$ all elements of $P$ are assigned a value by $f$ equal to one less than the length of the smallest maximal chain in which they and a minimal element of $P$ are contained,
	    therefore $f$ is a well defined pre-grading. 
	    The inductive construction will terminate at $h(P)-1$ since $f$ is uniquely determined on each maximal chain. 
	\end{proof}

    \begin{theorem}\label{thm:SemiExtension}
        Any semi-component module $S$ over a finite partial order $P$ arises as a summand of a component module $C$ in the following way
        \begin{equation*}
            C = S \oplus I,
        \end{equation*}
        where $I$ is an interval module over $P$.
    \end{theorem}
    
    \begin{proof}
        Given a semi-component module with a component basis, take its component extension with respect to this basis, giving us a component module and corresponding component basis.
        We now outline a procedure to decompose this component basis through elementary row and column operations on the corresponding matrices of its linear maps representing cover relations in the underlying partial order.
        
        Take some pre-grading of the partial order as defined in Definition \ref{def:Grading}, which exists by Lemma \ref{lem:Grading}.
        Beginning from cover relation matrices whose source is of pre-grade $0$ in the partial order, we may perform column operations subtracting the first column from every other column.
        Since these matrices must represent linear maps whose domain is a minimal element of the partial order, by Remark \ref{rmk:SymMarixDiag} these operations have no effect on other matrices.
        All zero columns of these matrices that were zero in the semi-component module are zero and the other columns have a $-1$ in the first row.
        These factors of $-1$ may be eliminated by row operations, as by definition there is exactly one $1$ in the entries below them.
        So if we add every row in turn to the top row, this makes the top row zero except for the first entry.
        Again by Remark \ref{rmk:SymMarixDiag}, this has the consequence of subtracting the first column to every other column in the matrices whose domain is the codomain of the matrices on which the row operations were performed. 
        
        By moving through the covering relations on the partial order with respect to their source value in the grading function, we may continue the above row operations ending at some maximal elements of the partial order in a consistent way.
        That is, on each matrix representing any given cover relation the order due to the pre-grading (see Definition \ref{def:Grading})
        ensures that all the column operations take place before the row operations. 
        
        The module now is written with respect to a basis where it can immediately be seen as the direct sum of the original semi-component module and an interval module,
        since all matrices have the same simultaneous block decomposition.
    \end{proof}
    
    \begin{remark}\label{rmk:EtensionUniquness}
        Given a $p\in P$ such that all generators lie in some $p\leq q\in P$,
         we can still derive a corresponding semi-component extension satisfying Theorem \ref{thm:SemiExtension} if Construction \ref{def:SemiExtension} were changed to leave all $s\leq p$ as zero dimensional vector space and the extension of Definition $\ref{def:SemiExtension}$ is applied as before to all $r \geq p$.
    \end{remark}
    
    Together Theorems \ref{thm:ComponentSplit} and \ref{thm:SemiExtension} show that semi-component modules play an important role in understanding the indecomposables of component modules as highlighted by the following corollary.
	
	\begin{corollary}\label{cor:Inverse}
	    Let $C$ be a component module over a partial  order $P$ with a minimal generator.
	    Then there is a unique decomposition
	    \begin{equation*}
            C = S \oplus I,
        \end{equation*}
        where $S$ is a semi-component module and $I$ an interval over $P$.
	    In particular $C$ is the component extension of $S$ for some choice of minimal generator in the sense of Remark \ref{rmk:EtensionUniquness}.
	\end{corollary}
	
	\begin{proof}
	    The interval module $I$ obtained in the proof of Theorem \ref{thm:ComponentSplit} is uniquely determined up to isomorphism as the module with dimension $1$ in every position greater than a minimal generator and identity maps between them.
	    Hence the uniqueness of $S$ follows from the Krull-Remak-Schmidt theorem.
	    
	    Using the proof of Theorem \ref{thm:SemiExtension} extension, the extension of $S$ with the minimal generator in the position it occurred in $C$
	    decomposes as $S \oplus I$ and is therefore $C$.
	\end{proof}

	In other words the corollary states under its conditions that up maintaining the position of the minimal generator the decomposition of theorem \ref{thm:ComponentSplit} and component extension are inverses.

    %%%%%%%%%%

    \subsection{Decompositions of semi-component modules}\label{sec:SemiDecomposition}

    Although Theorems \ref{thm:SemiExtension} and \ref{thm:ComponentSplit} show semi-component modules form an interesting sub-class of modules arsing from decompositions of component modules, they are not all indecomposable (see Example \ref{exam:IncomparableComponent}), do not contain all their indecomposable modules and still contain a large number of possible indecomposables.
    We first give an example showing that semi-component modules do decompose into two non-semi-component modules.
    
    \begin{example}\label{exam:SemiDecomposition}
        The following semi-component module has indecomposable decomposition,  
	    \begin{equation*}
    	    \begin{tikzpicture}[>=triangle 45]
                \node (A) at (0,0) {$\mathbb{K}^4$};
                \node (B) at (0,3) {$\mathbb{K}^2$};
                \node (C) at (3,0) {$\mathbb{K}^2$};
                \node (D) at (-3,0) {$\mathbb{K}^2$};
                \node (E) at (0,-3) {$\mathbb{K}^2$};
                \draw [semithick,->] (A) -- (B) node [right,midway] 
                {$\begin{bmatrix} 1 & 1 & 0 & 0 \\ 0 & 0 & 0 & 1 \end{bmatrix}$};
                \draw [semithick,->] (A) -- (C) node [midway,below]
                {$\begin{bmatrix} 0 & 0 & 1 & 1 \\ 0 & 1 & 0 & 0 \end{bmatrix}$};
                \draw [semithick,->] (A) -- (D) node [midway,above]
                {$\begin{bmatrix} 1 & 0 & 0 & 1 \\ 0 & 0 & 1 & 0 \end{bmatrix}$};
                \draw [semithick,->] (A) -- (E) node [left,midway]
                {$\begin{bmatrix} 0 & 1 & 1 & 0 \\ 1 & 0 & 0 & 0 \end{bmatrix}$};
                \node (Z) at (3.5,0) {$=$};
                \node (A') at (6.5,0) {$\mathbb{K}^2$};
                \node (B') at (6.5,1.5) {$\mathbb{K}$};
                \node (C') at (8,0) {$\mathbb{K}$};
                \node (D') at (4,0) {$\mathbb{K}$};
                \node (E') at (6.5,-1.5) {$\mathbb{K}$};
                \draw [semithick,->] (A') -- (B') node [right,midway] 
                {$\begin{bmatrix} 1 & 1 \end{bmatrix}$};
                \draw [semithick,->] (A') -- (C') node [midway,below]
                {$\begin{bmatrix} 0 & 1 \end{bmatrix}$};
                \draw [semithick,->] (A') -- (D') node [midway,above]
                {$\begin{bmatrix} \frac{1+\sqrt{5}}{2} & 1 \end{bmatrix}$};
                \draw [semithick,->] (A') -- (E') node [left,midway]
                {$\begin{bmatrix} 1 & 0 \end{bmatrix}$};

                \node (X) at (8.5,0) {$\oplus$};
                \node (A'') at (10.5,0) {$\mathbb{K}^2$};
                \node (B'') at (10.5,1.5) {$\mathbb{K}$};
                \node (C'') at (12,0) {$\mathbb{K}.$};
                \node (D'') at (9,0) {$\mathbb{K}$};
                \node (E'') at (10.5,-2) {$\mathbb{K}$};
                \draw [semithick,->] (A'') -- (B'') node [right,midway] 
                {$\begin{bmatrix} 1 & 1 \end{bmatrix}$};
                \draw [semithick,->] (A'') -- (C'') node [midway,below]
                {$\begin{bmatrix} 0 & 1 \end{bmatrix}$};
                \draw [semithick,->] (A'') -- (D'') node [midway,above]
                {$\begin{bmatrix} 1 & 0 \end{bmatrix}$};
                \draw [semithick,->] (A'') -- (E'') node [left,midway]
                {$\begin{bmatrix} \frac{1+\sqrt{5}}{2} & 1 \end{bmatrix}$};
            \end{tikzpicture}
        \end{equation*}
        Both of the decomposed modules can easily be seen to be incomparable and non-semi-component by the argument given in the example of \cite[\S 5.2]{Carlsson2009}.
        Using Lemma \ref{lem:DecompositionCharectersiation} the decomposition is obtained using an idempotent endomorphism given by linear endomorphisms
        \begin{equation*}
            A=\frac{5}{\sqrt{5}}
            \begin{bmatrix} 
                \frac{1+\sqrt{5}}{2} & 1 & 1 & 0 \\ 
                0 & \frac{\sqrt{5}-1}{2} & -1 & -1 \\
                1 & 0 & \frac{1+\sqrt{5}}{2} & 1 \\
                -1 & -1 & 0 & \frac{\sqrt{5}-1}{2}
            \end{bmatrix}
            \;\;\;\;\;\;\;
            B=\frac{5}{\sqrt{5}}
            \begin{bmatrix} 
                \frac{1+\sqrt{5}}{2} & -1 \\ 
                -1 & \frac{\sqrt{5}-1}{2}
            \end{bmatrix}
            \;\;\;\;\;\;\;
            C=\frac{5}{\sqrt{5}}
            \begin{bmatrix} 
                \frac{\sqrt{5}-1}{2} & 1 \\ 
                1 & \frac{1+\sqrt{5}}{2}
            \end{bmatrix}
            ,
        \end{equation*}
        where the first matrix is for the central vector space, the second matrix for the top and bottom vector spaces and the last matrix for the left and right vector spaces.
        The choice of basis of the kernel and image of $A$ are
        \begin{equation*}
            \Bigg{\langle} 
            \begin{bmatrix} 1 \\ 0 \\ -\frac{1+\sqrt{5}}{2} \\ \frac{1+\sqrt{5}}{2} \end{bmatrix},
            \begin{bmatrix} 0 \\ 1 \\ -1 \\ \frac{1+\sqrt{5}}{2} \end{bmatrix}
            \Bigg{\rangle} \text{  and  } \Bigg{\langle}
            \begin{bmatrix} \frac{1+\sqrt{5}}{2} \\ 0 \\ 1 \\ -1 \end{bmatrix},
            \begin{bmatrix} 1 \\ \frac{\sqrt{5}-1}{2} \\ 0 \\ -1 \end{bmatrix}
            \Bigg{\rangle},
        \end{equation*}
        respectively.
        The chosen bases of the four kernels and the four images of $B$ and $C$ in order top, left, bottom and right are
        \begin{equation*}
            \begin{bmatrix} 1 \\ \frac{1+\sqrt{5}}{2} \end{bmatrix},\;
            \begin{bmatrix} \frac{\sqrt{5}-1}{2} \\ 1 \end{bmatrix},\;
            \begin{bmatrix} -\frac{1+\sqrt{5}}{2} \\ 1  \end{bmatrix},\;
            \begin{bmatrix} \frac{1+\sqrt{5}}{2} \\ -1 \end{bmatrix} \text{  and  }
            \begin{bmatrix} \frac{1+\sqrt{5}}{2} \\ -1  \end{bmatrix},\;
            \begin{bmatrix} -1 \\ \frac{\sqrt{5}-1}{2} \end{bmatrix},\;
            \begin{bmatrix} \frac{\sqrt{5}-1}{2} \\ 1 \end{bmatrix}, \; 
            \begin{bmatrix} \frac{\sqrt{5}-1}{2} \\ 1 \end{bmatrix}
        \end{equation*}
        respectively.
    \end{example}
    
    The next example demonstrates that even in a relatively simple case there can still be infinitely many indecomposable semi-component modules over a given partial order.
    
    \begin{example}
        Let us consider the following family of semi-component representations of the $4$-star quiver:
        \begin{equation*}
        \begin{tikzpicture}[>=triangle 45]
        \node (A) at (0,0) {$\mathbb{K}^{2m}$};
        \node (B) at (0,3) {$\mathbb{K}^m$};
        \node (C) at (3,0) {$\mathbb{K}^m$};
        \node (D) at (-3,0) {$\mathbb{K}^m$};
        \node (E) at (0,-3) {$\mathbb{K}^m$};
        \draw [semithick,->] (A) -- (B) node [right,midway] {$\begin{bmatrix} I & 0  \end{bmatrix}$};
        \draw [semithick,->] (A) -- (C) node [midway,below] {$\begin{bmatrix} 0 & I \end{bmatrix}$};
        \draw [semithick,->] (A) -- (D) node [midway,above]
        {$\begin{bmatrix} I & I \end{bmatrix}$};
        \draw [semithick,->] (A) -- (E) node [left,midway]
        {$\begin{bmatrix} I & J_m(\lambda) \end{bmatrix}$};
        \end{tikzpicture}
        \end{equation*}
        where $J_m(\lambda)$ is the Jordan matrix
        \begin{equation*}
        \begin{bmatrix} \lambda & 1 & 0 & \cdots & 0 \\
        0 & \lambda & 1 & \cdots & 0 \\
        \vdots & \vdots & \vdots & \ddots & \vdots \\
        0 & 0 & 0 & \cdots & 1 \\
        0 & 0 & 0 & \cdots & \lambda
        \end{bmatrix},
        \end{equation*}
        and we use the value $\lambda = 0$ so that our representations are all semi-component modules. Then it is easy to show that all these representations are indecomposable for every $m\in \mathbb{N}$. This shows that, in general, there are infinitely many semi-component representations of quivers that are not of the Dynkin type.
        In fact, all indecomposables of this quiver can be classified, see \cite{Buchet2018}.
    \end{example}
    
    %%%%%%%%%%

	\bibliographystyle{plain}  
	\bibliography{References}	
	
\end{document}